\documentclass[letterpaper, 10 pt, conference]{ieeeconf}  

\IEEEoverridecommandlockouts                              

\usepackage{graphics} 
\usepackage{epsfig} 
\usepackage[cmex10]{amsmath}   
\usepackage{amsmath} 
\usepackage{amsfonts,amssymb,amsbsy}
\usepackage{latexsym,theorem}
\usepackage{epstopdf}
\usepackage{cite}
\usepackage{enumerate}

\newcommand{\Rset}{\mathbb{R}}
\newcommand{\Nset}{\mathbb{N}}
\newcommand{\cB}{\mathcal{B}}

\newcommand{\cD}{\mathcal{D}}
\newcommand{\cL}{\mathcal{L}}

{\theorembodyfont{\slshape}\newtheorem{theorem}{Theorem}[section]}
{\theorembodyfont{\slshape}\newtheorem{proposition}[theorem]{Proposition}}
{\theorembodyfont{\slshape}}
{\theorembodyfont{\slshape}}
{\theorembodyfont{\slshape}}
{\theorembodyfont{\upshape}}
{\theorembodyfont{\upshape}\newtheorem{problem}[theorem]{Problem}}
{\theorembodyfont{\upshape}\newtheorem{remark}[theorem]{Remark}}
{\theorembodyfont{\upshape}\newtheorem{assumption}[theorem]{Assumption}}

\title{\LARGE \bf
A Method to Guarantee Local Convergence for Sequential Quadratic Programming with Poor Hessian Approximation
}

\author{Tuan T. Nguyen, Mircea Lazar and Hans Butler
\thanks{The authors are with the Department of Electrical Engineering,
	Eindhoven University of Technology, P.O. Box 513, 5600
	MB Eindhoven, The Netherlands.
	E-mails: {\tt\small \{t.t.nguyen,m.lazar,h.butler\}@tue.nl}}%
}

\begin{document}

\maketitle
\thispagestyle{empty}
\pagestyle{empty}

\begin{abstract}
Sequential Quadratic Programming (SQP) is a powerful class of algorithms for solving nonlinear optimization problems. Local convergence of SQP algorithms is guaranteed when the Hessian approximation used in each Quadratic Programming subproblem is close to the true Hessian. However, a good Hessian approximation can be expensive to compute. Low cost Hessian approximations only guarantee local convergence under some assumptions, which are not always satisfied in practice. To address this problem, this paper proposes a simple method to guarantee local convergence for SQP with poor Hessian approximation. The effectiveness of the proposed algorithm is demonstrated in a numerical example.
\end{abstract}


\section{INTRODUCTION}
Sequential Quadratic Programming (SQP) is one of the most effective methods for solving nonlinear optimization problems. The idea of SQP is to iteratively approximate the Nonlinear Programming (NLP) problem by a sequence of Quadratic Programming (QP) subproblems~\cite{BoggsSQP1995}. The QP subproblems should be constructed in a way that the resulting sequence of solutions converges to a local optimum of the NLP.


There are different ways to construct the QP subproblems. When the exact Hessian is used to construct the QP subproblems, local convergence with quadratic convergence rate is guaranteed. However, the true Hessian can be indefinite when far from the solution. Consequently, the QP subproblems are non-convex and generally difficult to solve, since the objective may be unbounded below and there may be many local solutions~\cite{Gill2015}. Moreover, computing the exact Hessian is generally expensive, which makes SQP with exact Hessian difficult to apply to large-scale problems and real-time applications. 

To overcome these drawbacks, positive (semi-) definite Hessian approximations are usually used in practice. SQP methods using Hessian approximations generally guarantee local convergence under some assumptions. Some SQP variants employ iterative updates scheme for the Hessian approximation to keep it close to the true Hessian. Broyden-Fletcher-Goldfarb-Shanno (BFGS) is one of the most popular update schemes of this type~\cite{Powell1978,Fletcher1987}. The BFGS-SQP version guarantees superlinear convergence when the initial Hessian estimate is close enough to the true Hessian~\cite{BoggsSQP1995}. Another variant which is very popular for constrained nonlinear least square problems is the Generalized Gauss-Newton (GGN) method~\cite{Bock1983,Bock2015}. GGN method converges locally only if the residual function is small at the solution~\cite{DiehlThesis}. Some other SQP variants belong to the class of Sequential Convex Programming (SCP), or Sequential Convex Quadratic Programming (SCQP) methods, which exploit convexity in either the objective or the constraint functions to formulate convex QP subproblems~\cite{TranDinh2012,VerschuerenSCQP2016}. SCP methods also have local convergence under similar assumption of small residual function. However, these assumptions are not always satisfied in practice, resulting in poor Hessian approximation and thus no convergence is guaranteed.


This paper proposes a simple method to guarantee local convergence for SQP methods with poor Hessian approximations. The proposed method interpolates between the search direction provided by solving the QP subproblem and a feasible search direction. It is proven that there exists a suitable interpolation coefficient such that the resulting algorithm converges locally to a local optimum of the NLP with linear convergence rate. A numerical example is presented to demonstrate the effectiveness of the proposed method.

The idea of interpolating an optimal search direction with a feasible search direction was proposed in our previous work for quadratic optimization problems with nonlinear equality constraints~\cite{NguyenTTCDC2016}. The method proposed in~\cite{NguyenTTCDC2016} was applied effectively to a practical application in commutation of linear motors~\cite{NguyenTTIFAC2016}. This paper extends the idea to general nonlinear programming problems.

The remainder of this paper is organized as follows. Section~\ref{sec:notation} introduces the notation used in the paper. Section~\ref{sec:basicSQP} reviews the basic SQP method. Section~\ref{sec:proposedmethod} presents the proposed algorithm and proves the optimality property and local convergence property of the algorithm. An example is shown in Section~\ref{sec:example} for demonstration. Section~\ref{sec:conclusions} summarizes the conclusions.

\section{NOTATION}\label{sec:notation}
Let $\Nset$ denote the set of natural numbers, $\mathbb{R}$ denote the set of real numbers. The notation $\mathbb{R}_{[c_1,c_2)}$ denotes the set $\{ c\in \mathbb{R} : c_1 \leq c < c_2\}$. Let $\mathbb{R}^{n}$ denote the set of real column vectors of dimension $n$, $\mathbb{R}^{n \times m}$ denote the set of real $n \times m$ matrices. For a vector $x \in \Rset^n$, $x_{[i]}$ denotes the $i$-th element of $x$. The notation $0_{n \times m}$ denotes the $n \times m$ zero matrix and $I_n$ denotes the $n \times n$ identity matrix. Let $\|\cdot\|_2$ denote the 2-norm. The Nabla symbol $\nabla$ denotes the gradient operator. For a vector $x \in \Rset^n$ and a mapping $\Phi:\Rset^n \rightarrow \Rset$
\begin{equation*}
\nabla_x \Phi(x) = \begin{bmatrix}
\frac{\partial \Phi(x)}{\partial x_{[1]}} & \frac{\partial \Phi(x)}{\partial x_{[2]}} & \ldots & \frac{\partial \Phi(x)}{\partial x_{[n]}}
\end{bmatrix} .
\end{equation*}
Let $\cB (x_0, r)$ denote the open ball $\{x \in \Rset^n : \|x -x_0 \|_2 < r \}$.

\section{THE BASIC SQP METHOD}\label{sec:basicSQP}
This section reviews the basic SQP method. Consider the nonlinear optimization problem with nonlinear equality constraints:
\begin{problem}[NLP]\label{prb:optimization_prob}
	\begin{IEEEeqnarray}{rrl}
		& \min_{x} ~~~ & F_1(x)  \nonumber\\
		&\text{subject to} ~~~ & F_2(x)=0_{m \times 1} , \nonumber
	\end{IEEEeqnarray}
\end{problem}
where $x \in \mathbb{R}^n$, $F_1:\mathbb{R}^n \rightarrow \mathbb{R}$ and $F_2 : \mathbb{R}^n \rightarrow\mathbb{R}^m$. Here, $n$ is the number of optimization variables and $m$ is the number of constraints. In this paper, we are only interested in the case when the constraint set has an infinite number of points, i.e. $m<n$, since the other cases are trivial. Furthermore, let us assume that the columns of $\nabla_x F_2(x)^T$ are linearly independent at the solutions of the NLP.

For ease of presentation, in this paper we only consider equality constraints. The method can be extended to inequality constraints using an active set strategy or squared slack variables~\cite[Section 4]{BoggsSQP1995}. 

First, let us define the Lagrangian function of the NLP Problem~\ref{prb:optimization_prob} 
\begin{equation} \label{eqn:Lagrangian}
\cL(x,\lambda) := F_1(x) + \lambda^T F_2(x) ,
\end{equation}
where $\lambda\in \mathbb{R}^m$ is the Lagrange multipliers vector. The Karush-Kuhn-Tucker (KKT) optimality conditions of Problem~\ref{prb:optimization_prob} are
\begin{equation}\label{eqn:KKTcons}
\resizebox{.99\hsize}{!}{$
\nabla_{[x , \lambda]^T} \cL(x_*,\lambda_*)^T = \begin{bmatrix}
J_1(x_*)^T + J_2(x_*)^T \lambda_* \\
F_2(x_*)
\end{bmatrix} = 0_{(n+m)\times 1},
$}
\end{equation}  
where $J_1(x) := \nabla_x F_1(x)$ and $J_2(x) := \nabla_x F_2(x)$ are the Jacobian matrices of~$\nabla_x F_1(x)$ and~$\nabla_x F_2(x)$. Note that $J_1(x) \in \Rset^{1 \times n}$ and $J_2(x) \in \Rset^{m \times n}$.

The solution of the optimization problem is searched for in an iterative way. At a current iterate $x_k$, the next iterate is computed as
\begin{equation} \label{eqn:update}
x_{k+1} = x_k + \Delta x_k ,
\end{equation}
where $\Delta x_k$ is the search direction. In SQP methods, the search direction $\Delta x^{SQP}_k$ is the solution of the following QP subproblem
\begin{problem}[QP Subproblem]\label{prb:QPsub_prob}
	\begin{IEEEeqnarray}{rrl}
		& \min_{\Delta x_k} ~~~ & \frac{1}{2} \Delta x_k^T B_k \Delta x_k + J_{1|k} \Delta x_k \nonumber\\
		&\text{subject to} ~~~ & J_{2|k} \Delta x_k + F_{2|k}=0_{m \times 1} , \nonumber
	\end{IEEEeqnarray}
\end{problem}
where we introduce the following notation for brevity
\begin{align*}
F_{1|k} &:= F_1(x_k) , ~ F_{2|k} := F_2(x_k) \\
J_{1|k} &:= J_1(x_k) ,~ J_{2|k} := J_2(x_k) .
\end{align*}
Here, $B_k \in \Rset^{n \times n}$ is either the exact Hessian of the Lagrangian $\nabla_{xx}^2 \cL (x,\lambda)$, or a positive (semi-) definite approximation of the Hessian. Similar to~\cite{BoggsSQP1995}, to guarantee that the QP subproblem has a unique solution, we assume that the matrices $B_k$ satisfy the following conditions:
\begin{assumption} \label{asu:Bkpositive}
	The matrices $B_k$ are uniformly positive definite on the null spaces of the matrices $J_{2|k}$, i.e., there exists a $\beta_1>0$ such that for each $k$
	\begin{equation*}
	d^T B_k d \geq \beta_1 \|d\|_2^2 ,
	\end{equation*}
	for all $d\in \Rset^n$ which satisfy
	\begin{equation*}
	J_{2|k} d = 0_{m \times 1} .
	\end{equation*}
\end{assumption}

\begin{assumption} \label{asu:Bkbounded}
	The sequence $\{B_k\}$ is uniformly bounded, i.e, there exists a $\beta_2>0$ such that for each $k$
	\begin{equation*}
	\|B_k\|_2 \leq \beta_2.
	\end{equation*}
\end{assumption}


The KKT optimality conditions of the QP subproblem~\ref{prb:QPsub_prob} are
\begin{IEEEeqnarray}{rrl} 
& & B_k \Delta x^{SQP}_k + J_{1|k}^T + J_{2|k}^T \lambda_{k+1} = 0_{n\times 1} , \label{eqn:QPKKT1} \\
& & J_{2|k} \Delta x^{SQP}_k + F_{2|k} = 0_{m \times 1} , \label{eqn:QPKKT2} 
\end{IEEEeqnarray}
or equivalently
\begin{equation}   \label{eqn:QPKKT}
\begin{bmatrix}
B_k & J_{2|k}^T \\ J_{2|k} & 0_{m\times m}
\end{bmatrix} \begin{bmatrix}
\Delta x^{SQP}_k \\ \lambda_{k+1}
\end{bmatrix} = - \begin{bmatrix}
J_{1|k}^T \\ F_{2|k}
\end{bmatrix} .
\end{equation}
It should be noted that $B_k$ is positive (semi-) definite and is not necessarily invertible, but the matrix $\begin{bmatrix}
B_k & J_{2|k}^T \\ J_{2|k} & 0_{m\times m}
\end{bmatrix}$ is invertible due to Assumption~\ref{asu:Bkpositive}~\cite[Theorem 3.2]{Benzi2005}. Therefore, the KKT condition~\eqref{eqn:QPKKT} has a unique solution
\begin{equation}   \label{eqn:solutionQPKKT}
 \begin{bmatrix}
		\Delta x^{SQP}_k \\ \lambda_{k+1}
	\end{bmatrix} = - 	\begin{bmatrix}
	B_k & J_{2|k}^T \\ J_{2|k} & 0_{m\times m}
\end{bmatrix}^{-1} \begin{bmatrix}
		J_{1|k}^T \\ F_{2|k}
	\end{bmatrix} .
\end{equation}
For convergence analysis, it is convenient to have an explicit expression of $\Delta x^{SQP}_k$. Since $J_{2|k}^T J_{2|k}$ is positive semidefinite and $B_k$ is positive deinite on the null space of $J_{2|k}$, there exists a constant $c_0$ such that~\cite[Lemma 3.2.1]{Bertsekasbook1999}
\begin{equation}
B_k + c J_{2|k}^T J_{2|k} \succ 0 ,~ \forall c > c_0 .
\end{equation}
Let us define
\begin{IEEEeqnarray*}{rrl} 
C_k &:=& B_k + \bar{c} J_{2|k}^T J_{2|k} , \text{ where } \bar{c}> c_0 , \\
D_k &:=& J_{2|k} C_k^{-1} J_{2|k}^T .
\end{IEEEeqnarray*}
We have that $C_k$ and $D_k$ are positive definite due to Assumption~\ref{asu:Bkpositive}. It holds that~\cite[Chapter 6]{Bernstein2009matrix}
\begin{IEEEeqnarray}{rrl} \label{eqn:invSaddlePointMatrix}
&& \begin{bmatrix}
B_k & J_{2|k}^T \\ J_{2|k} & 0_{m\times m}
\end{bmatrix}^{-1} \nonumber\\
&=& \begin{bmatrix}
C_k^{-1} - C_k^{-1} J_{2|k}^T D_k^{-1} J_{2|k} C_k^{-1} & C_k^{-1} J_{2|k}^T D_k^{-1} \\ (C_k^{-1} J_{2|k}^T D_k^{-1})^T & I_m - D_k^{-1}
\end{bmatrix} .
\end{IEEEeqnarray}
The solution $\Delta x^{SQP}_k$ can then be written in an explicit form
\begin{equation}   \label{eqn:SQPDx}
	\Delta x^{SQP}_k = -\left(I_n - T^{C_k}_{2|k} J_{2|k} \right) C_k^{-1} J_{1|k}^T - T^{C_k}_{2|k} F_{2|k} ,
\end{equation}
where 
\begin{equation*} 
	T^{C_k}_{2|k} := C_k^{-1} J_{2|k}^T (J_{2|k} C_k^{-1} J_{2|k}^T )^{-1} .
\end{equation*}
Notice that $T^{C_k}_{2|k} \in \Rset^{n \times m}$ is a generalized right inverse of $J_{2|k}$, i.e. $J_{2|k} T^{C_k}_{2|k} = I_m$.

It should be noted that if $B_k$ is nonsingular then $\Delta x^{SQP}_k$ can also be written as
\begin{equation}   \label{eqn:SQPDxBnonsingular}
\Delta x^{SQP}_k = -\left(I_n - T^{B_k}_{2|k} J_{2|k} \right) B_k^{-1} J_{1|k}^T - T^{B_k}_{2|k} F_{2|k} ,
\end{equation}
where
\begin{equation*} 
	T^{B_k}_{2|k} := B_k^{-1} J_{2|k}^T (J_{2|k} B_k^{-1} J_{2|k}^T )^{-1} .
\end{equation*}
In this case, both~\eqref{eqn:SQPDx} and~\eqref{eqn:SQPDxBnonsingular} give the same solution.

If $B_k$ is the exact Hessian then the basic SQP method is equivalent to applying Newton's method to solve the KKT conditions~\eqref{eqn:KKTcons}, which guarantees quadratic local convergence rate~\cite[Chapter 18]{Nocedalbook}. When an approximation is used instead, local convergence is guaranteed only when $B_k$ is close enough to the true Hessian. The readers are referred to~\cite[Section 3]{BoggsSQP1995} for more details on local convergence of SQP.

\section{PROPOSED METHOD}\label{sec:proposedmethod}
This section proposes a simple method to guarantee local convergence for SQP with poor Hessian approximation. The proposed method interpolates between an optimal search iteration, without local convergence guarantee, and a feasible search iteration with guaranteed local convergence.

The search direction $\Delta x^{SQP}_k$ can be viewed as the optimal direction which iteratively leads to the optimal solution of the NLP, if the iteration converges. However, local convergence is not guaranteed if $B_k$ is a poor approximation of the true Hessian. 

To guarantee local convergence with poor Hessian approximation, we propose a new search direction which is the interpolation between the optimal search direction $\Delta x^{SQP}_k$ and a feasible search direction $\Delta x^{f}_k$, i.e.
\begin{equation} \label{eqn:finalDx}
\Delta x_k = \alpha \Delta x^{SQP}_k + (1-\alpha ) \Delta x^{f}_k ,
\end{equation}
where $\alpha\in\Rset_{(0,1)}$. The feasible search direction $\Delta x^{f}_k$ only searches for a feasible solution of the set of constraints, but its local convergence is guaranteed. The idea of this proposed interpolated update is to combine the optimality property of the SQP update and the local convergence property of the feasible update. 

The feasible search direction can be found as a solution of the linearized constraints
\begin{equation} \label{eqn:feasbilecons}
J_{2|k} \Delta x^{f}_k + F_{2|k} =0_{m \times 1} .
\end{equation}
Since $m<n$, there is an infinite number of solutions for~\eqref{eqn:feasbilecons}. Two possible solutions are
\begin{IEEEeqnarray}{rrl} 
\Delta x^{f1}_k & = & - T^{C_k}_{2|k} F_{2|k} , \label{eqn:feasibleDx1}\\
\Delta x^{f2}_k & = & - T_{2|k} F_{2|k} , \label{eqn:feasibleDx2}
\end{IEEEeqnarray}
where $T_{2|k} \in \Rset^{n \times m}$ is the Moore-Penrose generalized right inverse of~$J_{2|k}$~\cite{Rao1972}, i.e.
\begin{equation*} 
T_{2|k} :=  J_{2|k}^T (J_{2|k} J_{2|k}^T )^{-1} .
\end{equation*}
We propose the following feasible search direction 
\begin{equation} \label{eqn:feasbileDx}
\Delta x^{f}_k = - \left(\frac{1}{1-\alpha} T_{2|k} - \frac{\alpha}{1-\alpha}T^{C_k}_{2|k} \right) F_{2|k} .
\end{equation}
It can be verified that~$\Delta x^{f}_k$ is a solution of~\eqref{eqn:feasbilecons} as follows
\begin{IEEEeqnarray}{rrl}  
J_{2|k}^T \Delta x^{f}_k &=& - \left(\frac{1}{1-\alpha} J_{2|k}^T T_{2|k} - \frac{\alpha}{1-\alpha} J_{2|k}^T T^{C_k}_{2|k} \right) F_{2|k} \nonumber\\
&=& - \left(\frac{1}{1-\alpha} I_m - \frac{\alpha}{1-\alpha} I_m \right) F_{2|k} \nonumber \\
&=& -F_{2|k} .
\end{IEEEeqnarray}
It has been proven that the feasible updates~\eqref{eqn:feasibleDx1}, \eqref{eqn:feasibleDx2} and~\eqref{eqn:feasbileDx} converge locally to a feasible solution of the constraints~\cite{BENISRAEL1966,Levin2001}.

It is worth mentioning that using the search direction $\Delta x^{f2}_k$ in~\eqref{eqn:feasibleDx2} can also guarantee local convergence for the interpolated update. However, this search direction results in the presence of the term $\alpha T^{C_k}_{2|k} F_{2|k}$ in the interpolated update~\eqref{eqn:finalDx}, which unnecessarily increases the computational load. Therefore, the search direction $\Delta x^{f}_k$ in~\eqref{eqn:feasbileDx} is proposed to help eliminate the unnecessary term $\alpha T^{C_k}_{2|k} F_{2|k}$ from the interpolated update~\eqref{eqn:finalDx}.

Substituting~\eqref{eqn:SQPDx} and~\eqref{eqn:feasbileDx} into the interpolated update~\eqref{eqn:finalDx} results in
\begin{equation}  \label{eqn:finalupdate}
\Delta x_k = - \alpha \left( I_n - T^{C_k}_{2|k} J_{2|k} \right) C_k^{-1} J_{1|k}^T - T_{2|k} F_{2|k} .
\end{equation}
For brevity, let us denote $G: \mathbb{R}^n \rightarrow \mathbb{R}^n$ as follows
\begin{equation} \label{eqn:Gk}
G_k:= \left( I_n - T^{C_k}_{2|k} J_{2|k} \right) C_k^{-1} J_{1|k}^T.
\end{equation}

In what follows we will prove the optimality property and the local convergence property of the proposed search iteration~\eqref{eqn:finalupdate}.



\begin{theorem}   \label{thm:optimality}
	If the iteration~\eqref{eqn:finalupdate} converges to a fixed point $x_*$, then $x_*$ satisfies the KKT optimality conditions~\eqref{eqn:KKTcons}.
\end{theorem}

\begin{proof}  
Let us denote
\begin{IEEEeqnarray}{rrl} 
	&F_{1*} := F_1(x_*), ~ & F_{2*} := F_2(x_*) , \nonumber \\
	&J_{1*} := J_1(x_*), ~ & J_{2*} := J_2(x_*) . \nonumber
\end{IEEEeqnarray}
From~\eqref{eqn:QPKKT2}, \eqref{eqn:feasbilecons} and~\eqref{eqn:finalDx}, it follows that
\begin{equation} \label{eqn:feasible_condition}
J_{2|k} \Delta x_k + F_{2|k} =0_{m\times 1}  .
\end{equation}
By definition, $x_*$ is a fixed point of the proposed iteration~\eqref{eqn:finalupdate} if
\begin{equation} \label{eqn:Dxeq0}
\Delta x_* = 0_{n\times 1}.
\end{equation}
As a result we have
\begin{equation} \label{eqn:provedKKT2}
F_{2*} = 0_{m\times 1} .
\end{equation}
Substituting~\eqref{eqn:provedKKT2} into~\eqref{eqn:feasbileDx} results in
\begin{equation}  \label{eqn:Dxfeq0}
\Delta x^{f}_* = 0_{n\times 1}.
\end{equation}
It follows from~\eqref{eqn:finalDx}, \eqref{eqn:Dxeq0} and \eqref{eqn:Dxfeq0} that
\begin{equation}  \label{eqn:DxSQPeq0}
\Delta x^{SQP}_* = 0_{n\times 1}.
\end{equation}
Due to~\eqref{eqn:QPKKT1} and~\eqref{eqn:DxSQPeq0} we have
\begin{equation} \label{eqn:provedKKT1}
J_{1*}^T + J_{2*}^T \lambda_{*} = 0_{n\times 1} .
\end{equation}
From~\eqref{eqn:provedKKT2} and~\eqref{eqn:provedKKT1}, it can be concluded that $x_*$ satisfies the KKT optimality conditions~\eqref{eqn:KKTcons}.
\end{proof}

Next, we will prove local convergence of the proposed iteration. Let us assume that the approximations $B_k$ satisfy the following condition
\begin{assumption} \label{asu:BkLipschitz}
	There exists a $\beta_3>0$ such that for each $k\geq 1$
	\begin{equation*}
	\|B_{k} - B_{k-1} \|_2 \leq \beta_3 \|x_{k} - x_{k-1} \|_2 .
	\end{equation*}
\end{assumption}
The following proposition will be used in the proof.
\begin{proposition}\label{pro:meanvalue}
	Let $\cD \subseteq \Rset^n$ be a convex set in which $F_2:\cD \rightarrow \mathbb{R}^m$ is differentiable and $J_2(x)$ is Lipschitz continuous for all~$x \in \cD$, i.e. there exists a $\gamma>0$ such that
	\begin{equation}
	\| J_2(x)-J_2(y) \|_2 \leq 2 \gamma \|x-y \|_2 , \forall  x,y \in \cD .
	\end{equation}
	Then
	\begin{equation}
	\| F_2(x)-F_2(y)-J_2(y)(x-y) \|_2 \leq \gamma \|x-y \|_2^2 , \forall  x,y \in \cD .
	\end{equation}
\end{proposition}
A proof of Proposition~\ref{pro:meanvalue} can be found in~\cite{Levin2001}.


\begin{theorem} \label{thm:convergence}
Let $\cD \subset \Rset^n$ be a bounded convex set in which the following conditions hold
\begin{enumerate}[(i)]
\item $F_1(x)$ and $F_2(x)$ are Lipschitz continuous and continuosly differentiable,
\item $J_1(x)$ and $J_2(x)$ are Lipschitz continuous and bounded,
\item $T_2(x)$ is bounded,
\item there exists a solution $x_* $ of the KKT optimality conditions~\eqref{eqn:KKTcons} in~$\cD$.
\end{enumerate}
Then there exist a $\alpha \in\Rset_{(0,1)}$ and a $r \in \Rset_{>0}$ such that $\cB(x_*,r) \subseteq \cD$ and iteration~\eqref{eqn:finalupdate} converges to $x_* $ for any initial estimate $x_0 \in \cB(x_*,r)$.
\end{theorem}

\begin{proof}
Let us consider two cases
\begin{itemize}
	\item $F_2(x)$ is linear.
	\item $F_2(x)$ is nonlinear.
\end{itemize}

1. \textit{Case 1}: in the first case when $F_2(x)$ is linear, for any iterate $x_k$ we can write
\begin{equation}
F_2(x)= F_{2|k} + J_{2|k} (x-x_k) .
\end{equation}
Since $\Delta x_k$ satisfies~\eqref{eqn:feasible_condition}, it follows that \begin{equation}
F_{2|k+1} = F_{2|k} + J_{2|k} (x_{k+1}-x_k)  = 0_{m\times 1} .
\end{equation}
Therefore, we have that $F_{2|k}=0_{m\times 1}$ for all $k \geq 1$. The interpolated update is then reduced to
\begin{equation}  \label{eqn:finalupdatelinear}
x_{k+1}= x_{k} +\alpha \Delta x^{SQP}_{k} , ~\forall k\geq 1.
\end{equation}
Let us denote
\begin{equation}
W_k:=C_k^{-1} - C_k^{-1} J_{2|k}^T D_k^{-1} J_{2|k} C_k^{-1} .
\end{equation}
From~\eqref{eqn:solutionQPKKT} and \eqref{eqn:invSaddlePointMatrix} we have
\begin{equation}
\Delta x^{SQP}_{k} = - W_{k} J_{1|k}^T , ~\forall k \geq 1.
\end{equation}
We have $\begin{bmatrix}
B_k & J_{2|k}^T \\ J_{2|k} & 0_{m\times m}
\end{bmatrix}$ is positive definite due to Assumption~\ref{asu:Bkpositive}. It follows that $W_{k}$ is positive definite, due to the facts that the inverse of a positive definite matrix is positive definite, and that every principal submatrix of a positive definite matrix is positive definite~\cite[Chapter 8]{Bernstein2009matrix}. As a result we have
\begin{equation}
J_{1|k} \Delta x^{SQP}_{k} = - J_{1|k} W_{k} J_{1|k}^T < 0 , ~\forall k \geq 1.
\end{equation}
This shows that $\Delta x^{SQP}_{k} $ is a descent direction that leads to a decrease in the cost function $F_1(x)$. In addition, since $F_{2|k}=0_{m\times 1}$ for all $k \geq 1$, we have that $\Delta x^{SQP}_{k} $ is also a feasible direction. Therefore, there exits a stepsize $\alpha \in \Rset_{(0,1)}$ such that the iteration~\eqref{eqn:finalupdatelinear} converges~\cite[Chapter 3]{Nocedalbook}.

2. \textit{Case 2}: let us now consider the case when $F_2(x)$ is nonlinear. Since the nonlinear constraints are solved by successive linearization~\eqref{eqn:feasible_condition}, we can assume that the solution is reached asymptotically, i.e. $F_{2|k-1} \to 0$ as $k \to \infty$ and $F_{2|k-1} \neq 0$ for all $k< \infty$. We have
\begin{IEEEeqnarray}{rrl} \label{eqn:recurrence_relation}
	\Delta x_k &=&	\Delta x_{k-1 } - \alpha (G_{k}-G_{k-1}) \nonumber\\ 
	&&-\left(T_{2|k} F_{2|k} - T_{2|k-1} F_{2|k-1} \right) \nonumber\\
	&=&	\Delta x_{k-1 } - \alpha (G_{k}-G_{k-1}) \nonumber\\ 
	&&-T_{2|k} F_{2|k} - T_{2|k-1} J_{2|k-1} \Delta x_{k-1 } \nonumber\\ 
	&=&	\left(I_n - T_{2|k-1} J_{2|k-1} \right) \Delta x_{k-1 } - \alpha (G_{k}-G_{k-1}) \nonumber\\ 
	&&-T_{2|k} F_{2|k} .
\end{IEEEeqnarray}
The second equality in~\eqref{eqn:recurrence_relation} was obtained due to equation~\eqref{eqn:feasible_condition}. 

Let us consider the first term on the right hand side of~\eqref{eqn:recurrence_relation}. Here, $(I_n - T_{2|k-1} J_{2|k-1})$ is the orthogonal projection onto the null space of $J_{2|k-1}$~\cite[Chapter 6]{Bernstein2009matrix}. It holds that
\begin{equation} \label{eqn:normprojection}
\| I_n - T_{2|k-1} J_{2|k-1} \|_2 = 1 .
\end{equation}
A proof of~\eqref{eqn:normprojection} can be found in~\cite{NguyenTTCDC2016}. It follows that
\begin{equation} 
\| (I_n-T_{2|k-1} J_{2|k-1}) \Delta x_{k-1 } \|_2  \leq \| \Delta x_{k-1 } \|_2 .
\end{equation}
The equality holds if and only if $\Delta x_{k-1 }$ is in the null space of $J_{2|k-1}$, which is equivalent to
\begin{equation} 
J_{2|k-1} \Delta x_{k-1 } = -F_{2|k-1} =0_{m \times 1} .
\end{equation}
This shows that the equality holds if and only if $x_{k-1}$ is an exact solution of the constraints. This contradicts the assumption that $F_{2|k-1} \neq 0$ for all $k< \infty$. Therefore, there exists a constant $M\in \Rset_{(0,1)}$ such that
\begin{equation}\label{eqn:ineq1}
\| (I_n-T_{2|k-1} J_{2|k-1}) \Delta x_{k-1 }  \|_2 \leq M \| \Delta x_{k-1 } \|_2 .
\end{equation}

Next, let us consider the second term on the right hand side of~\eqref{eqn:recurrence_relation}. Observe that by~\eqref{eqn:Gk}, the definition of the matrix inverse and the strict positive definiteness of $C_k$, each element of $G_k$ is obtained by adding, multiplying and/or division of real-valued functions. Division only occurs due to the inverse of $C_k$, via the term $\frac{1}{\det (C_k)}$. This allows the application of Theorem 12.4 and Theorem 12.5 in~\cite{ErikssonBook2004}, to establish Lipschitz continuity in $x$ of $G_k$, from Assumptions~\ref{asu:Bkbounded} and \ref{asu:BkLipschitz} and the conditions that $J_1(x)$ and $J_{2}(x)$ are Lipschitz continuous and bounded for all $x\in \cD$. Note that although the theorems in~\cite{ErikssonBook2004} consider functions from $\Rset$ to $\Rset$, the same arguments apply to functions from $\Rset^n$ to $\Rset$, by using an appropriate, norm-based Lipschitz inequality. As a result we have
\begin{equation}\label{eqn:ineq2}
\| G_{k}-G_{k-1} \|_2 \leq N \| \Delta x_{k-1 } \|_2 ,
\end{equation}
where $N>0$.

For the third term on the right hand side of~\eqref{eqn:recurrence_relation}, due to the condition that $T_2(x)$ is bounded and Proposition~\ref{pro:meanvalue}, we have
\begin{IEEEeqnarray}{rrl} \label{eqn:ineq3} 
\| T_{2|k} F_{2|k} \|_2 &\leq& \| T_{2|k} \|_2 \|F_{2|k} \|_2 \nonumber \\
&= &  \| T_{2|k} \|_2 \| F_{2|k}-F_{2|k-1}-J_{2|k-1} \Delta x_{k-1}\|_2 \nonumber \\
&\leq& \gamma \| T_{2|k} \|_2 \|\Delta x_{k-1}\|_2^2 \nonumber \\
&\leq& L \|\Delta x_{k-1}\|_2^2 , 
\end{IEEEeqnarray}
where $L>0$.

From~\eqref{eqn:recurrence_relation}, \eqref{eqn:ineq1}, \eqref{eqn:ineq2} and \eqref{eqn:ineq3}, it follows that
\begin{IEEEeqnarray}{rrl} \label{eqn:total_ineq}
	\| \Delta x_{k}\|_2 &\leq& (K + L \|\Delta x_{k-1}\|_2)  \|\Delta x_{k-1}\|_2  ,
\end{IEEEeqnarray}
where~$K= M +\alpha N$. We have $K <1$ for any $\alpha$ that satisfies
\begin{equation}
0<\alpha <\min \left(\frac{1-M}{N},1 \right) .
\end{equation}

From~\eqref{eqn:finalupdate}, \eqref{eqn:Dxeq0} and \eqref{eqn:provedKKT2}, it follows that
\begin{equation}
G_* = 0_{n \times 1} .
\end{equation}
Due to the Lipschitz continuity of $G(x)$ and $F_2(x)$, we have
\begin{IEEEeqnarray}{rrl} 
\| \Delta x_0 \|_2 &=& \|-\alpha G_0 - T_{2|0} F_{2|0} \|_2 \nonumber \\
&=& \| -\alpha (G_0 - G_*) - T_{2|0} (F_{2|0} - F_{2*}) \|_2 \nonumber \\
& \leq & Q \|x_0 - x_*\|_2 ,
\end{IEEEeqnarray}
where $Q>0$. If $x_0$ is close enough to $x_*$ such that 
\begin{equation} \label{eqn:ballradius}
\|x_0 - x_*\|_2 < \frac{1-K}{Q L} ,
\end{equation}
then 
\begin{equation} \label{eqn:induction0}
K  + L \| \Delta x_0 \|_2 < 1 .
\end{equation}
Next, we will prove that if
\begin{equation} \label{eqn:induction1}
K  + L \|\Delta x_{k-1}\|_2 < 1 ,
\end{equation}
then
\begin{equation} \label{eqn:induction2}
K  + L \|\Delta x_{k}\|_2 < 1 ,~\forall k=1,2,\ldots.
\end{equation}
Indeed, if~\eqref{eqn:induction1} holds then due to~\eqref{eqn:total_ineq} we have
\begin{equation} 
\|\Delta x_{k}\|_2 < \|\Delta x_{k-1}\|_2 .
\end{equation}
This leads to
\begin{equation}
K  + L \|\Delta x_{k}\|_2 < K  + L \|\Delta x_{k-1}\|_2 < 1 .
\end{equation}
We have proven that if~\eqref{eqn:induction1} holds then~\eqref{eqn:induction2} holds. Since~\eqref{eqn:induction0} also holds for any $x_0$ which satisfies~\eqref{eqn:ballradius}, it follows by induction that
\begin{equation}
K  + L \|\Delta x_{k}\|_2 < 1 ,~ \forall k=1,2,\ldots.
\end{equation}
Therefore, it follows from~\eqref{eqn:total_ineq} that
\begin{equation}\label{eqn:convergenceproven}
\|\Delta x_{k}\|_2 < \|\Delta x_{k-1}\|_2 ,~ \forall k=1,2,\ldots.
\end{equation}
Therefore, algorithm~\eqref{eqn:finalupdate} converges, and by Theorem~\ref{thm:optimality}, it converges to a KKT point, for any initial estimate $x_0 \in \cB \left(x_*,r \right)$, where
\begin{equation}
r=\frac{1-K}{QL} ,
\end{equation}
and any $\alpha\in\Rset_{(0,1)}$ which makes $K<1$.
\end{proof}

It can be seen from~\eqref{eqn:convergenceproven} that the proposed algorithm has a linear convergence rate.

\begin{remark}
The explicit expression~\eqref{eqn:SQPDx} is of interest for convergence analysis. For implementation, instead of~\eqref{eqn:SQPDx}, the SQP search direction can also be computed as
\begin{equation}   \label{eqn:SQPDx2ndoption}
\Delta x^{SQP}_k = -\begin{bmatrix} I_n & 0_{n \times m}
\end{bmatrix} \begin{bmatrix} B_k & J_{2|k}^T \\ J_{2|k} & 0_{m\times m}
\end{bmatrix}^{-1} \begin{bmatrix}
J_{1|k}^T \\ F_{2|k}
\end{bmatrix} . 
\end{equation}
Note that~\eqref{eqn:SQPDx2ndoption} differs from~\eqref{eqn:SQPDx} in implementation, but they both give the same solution. In this case, using the feasible search direction $\Delta x^{f2}_k$ in~\eqref{eqn:feasibleDx2} for the interpolated iteration is more convenient. It can be proven in a similar way that the optimality property and local convergence property hold for the resulting interpolated iteration~\eqref{eqn:finalDx}.
\end{remark}

The proposed method can be applied to any positive (semi) definite Hessian approximations which satisfy Assumptions~\ref{asu:Bkpositive}, \ref{asu:Bkbounded}, \ref{asu:BkLipschitz}. Popular Hessian approximations such as GGN, or any constant Hessian approximation satisfy these conditions. It is worth noting that the simple identity approximation~$B_k=I_n$ also satisfies the mentioned conditions. 


The proposed method therefore can be useful in some of the following situations. When the exact Hessian is indefinite or is too expensive to compute and the search iteration using Hessian approximations fails to converge, the proposed method can be used to enforce convergence. For large-scale cases when even Hessian approximations are computationally costly, the simple identity Hessian approximation $B_k=I_n$ can be used together with the proposed interpolation method. This results in the same search iteration as proposed in~\cite{TorrisiCDC2016,TorrisiarXiv2016}, although the iteration and convergence therein were derived in a different way. Furthermore, if the cost function is just the 2-norm $F_1(x)=x^T x$ and the identity Hessian approximation is used then the proposed algorithm recovers the algorithm in our previous work~\cite{NguyenTTCDC2016}. It should be noted, however, that the identity Hessian approximation may result in a slower convergence rate compared to other Hessian approximations, as can be seen in the example in Section~\ref{sec:example}.

\section{NUMERICAL EXAMPLE}\label{sec:example}
This section presents a numerical example to verify the performance
of the proposed algorithm.
Let us consider the test problem 77 in~\cite{HockTest1980}.
\begin{problem}\label{prb:Hock77}
 	\begin{IEEEeqnarray}{rrl}
 		& \min_{x\in \Rset^5} ~~~ & (x_1-1)^2+(x_1-x_2)^2\nonumber\\
 		&&+(x_3-1)^2+(x_4-1)^4+(x_5-1)^6 \nonumber\\
 		&\text{subject to} ~~~ & x_1^2 x_4 +\sin(x_4-x_5) -2\sqrt{2}=0 , \nonumber \\
 		&& x_2 +x_3^4 x_4^2-8-\sqrt{2} =0 . \nonumber
 	\end{IEEEeqnarray}
\end{problem}
The initial estimate is $x_0=\begin{bmatrix}
2 & 2 & 2&2&2
\end{bmatrix}^T$ and $\lambda_0=\begin{bmatrix}
0 & 0 
\end{bmatrix}^T$. This is a nonlinear equality constrained least square problem with nonzero residual. 

In nonlinear constrained least square problems, the cost function has the least square form 
\begin{equation*}
F_1(x)= \frac{1}{2} \| R(x) \|_2^2 ,
\end{equation*}
where $R : \mathbb{R}^n \rightarrow\mathbb{R}^p$. A popular Hessian approximation for this type of problems is the GGN approximation
\begin{equation*}
B^{GGN} (x) = \nabla_x R(x)^T \nabla_x R(x).
\end{equation*}
It is well known that the SQP method with GGN Hessian approximation, also called the GGN method, converges locally if the residual function $R(x)$ is small at the solution~\cite{DiehlThesis}.

In this example, we test the exact Hessian SQP method (SQP-EH), the GGN method (SQP-GGN), the proposed interpolated method with GGN Hessian approximation (iSQP-GGN), and the proposed interpolated method with identity Hessian approximation $B_k=I_n$ (iSQP-I). The optimization algorithms are programmed in Matlab and tested on a 2.4GHz computer. The measure of convergence is the 2-norm of the KKT matrix~\eqref{eqn:KKTcons}, which is called the KKT residual. The optimization algorithms terminate when the KKT residual is less than $10^{-7}$.

The test results are as follows. The SQP-EH method converges quadratically as expected. The SQP-GGN method does not converge. The iSQP-GGN method converges linearly. This demonstrates that the proposed interpolation scheme can guarantee convergence for the GGN Hessian approximation. The iSQP-I method also converges linearly, but at a slower rate. This is expected since the GGN approximation is a better approximation than the identity matrix. The convergence rate of the methods are shown in Fig~\ref{fig:convergence_rate}. The interpolation coefficients $\alpha$ shown here are among the ones that result in fastest convergence rates for each method.

\begin{figure}[t]
	\centering
	\includegraphics[width=1\columnwidth]{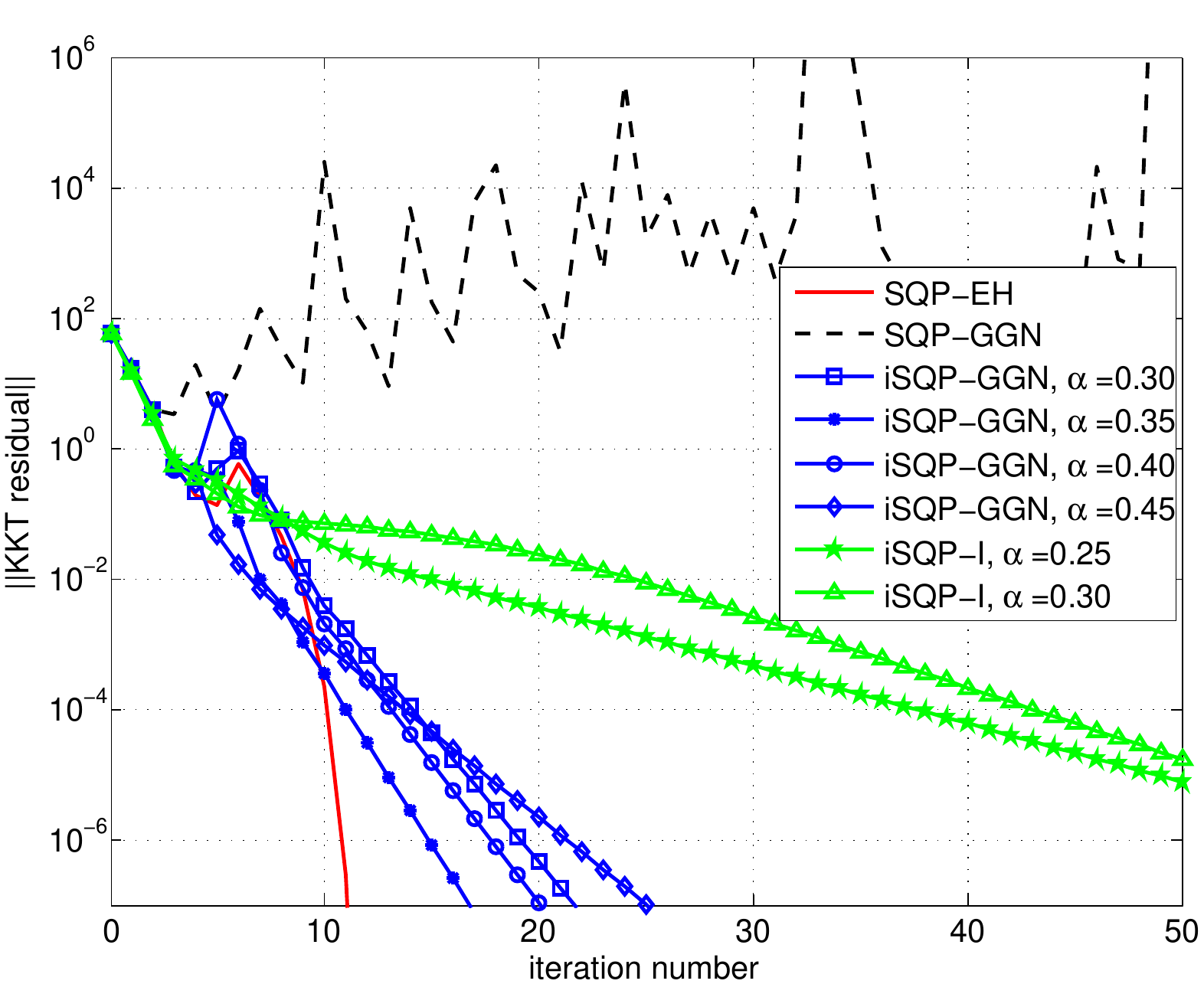}
	\caption{Convergence rate illustration.}
	\label{fig:convergence_rate}
\end{figure}


The SQP-EH method, the proposed iSQP-GGN and iSQP-I methods converge to the same solution $x = [1.166172,1.182111,1.380257,1.506036,0.610920]^T$, which is the same with the solution mentioned in~\cite{HockTest1980}. 

The number of iterations and computation times are summarized in Table~\ref{tbl:computation_times}. It is observed that the SQP-EH method requires the least number of iterations, as it converges quadratically. The iSQP-GGN method with $\alpha=0.35$ needs a larger number of iterations, but the total computation time is lower, since it requires less computation per iteration. This demonstrates that with a suitable choice of $\alpha$, the proposed method can be more efficient than the SQP-EH method, especially in large-scale cases when computation of the exact Hessian can be very expensive. 

\begin{table}[h]
	\caption{Computation times}
	\label{tbl:computation_times}
	\begin{center}
		\begin{tabular}{|l|c|c|c|}
			\hline
			\textit{Method} & \textit{Number of} & \textit{Computation} \\
			\textit{} & \textit{iterations} & \textit{time (ms)} \\
			\hline
			SQP-EH & $13$ & $113$ \\
			iSQP-GGN, $\alpha=0.30$ & $23$ & $141$ \\
			iSQP-GGN, $\alpha=0.35$ & $18$ & $107$ \\
			iSQP-GGN, $\alpha=0.40$ & $22$ & $136$ \\
			iSQP-GGN, $\alpha=0.45$ & $27$ & $160$ \\
			iSQP-I, $\alpha=0.25$ & $73$ & $359$ \\
			iSQP-I, $\alpha=0.30$ & $72$ & $347$ \\
			\hline
		\end{tabular}
	\end{center}
\end{table}

Examples of large-scale problems are nonlinear model predictive control (NMPC) problems. In~\cite{TorrisiCDC2016,TorrisiarXiv2016}, the iSQP-I method, which is called projected gradient and constraint linearization method therein, is shown to outperform some commercial solvers when applying it to the NMPC problem for an inverted pendulum. The results of the example above suggest that with a suitable choice of the Hessian approximation, e.g. GGN approximation, the proposed method may even perform better, given the special sparse structure of the NMPC problem. Demonstrating this will be a subject of our future research.



\section{CONCLUSIONS} \label{sec:conclusions}
This paper proposed a method to guarantee local convergence for SQP with poor Hessian approximation. The proposed method interpolates between the SQP search direction and a suitable feasible search direction, in order to combine the optimality property and the local convergence property of the two search directions. It was proven that the proposed algorithm converges locally at linear rate to a KKT point of the nonlinear programming problem. The effectiveness of the method was illustrated in a numerical example.

In this paper we only consider the local convergence property. For future work, we will extend the convergence result to global convergence using an augmented Lagrangian merit function~\cite{Gill1992}.






%
%

\bibliographystyle{IEEEtran} 
\bibliography{references}

\end{document}